\documentclass[a4paper,12pt]{amsart}
\usepackage[T1]{fontenc}    
\usepackage[utf8]{inputenc}
\usepackage{enumerate}
\usepackage{amssymb}
\usepackage{fullpage}

\def\eps{\varepsilon }
\def\RR{\mathbb R}
\def\mm{\mathcal M}

\newcommand{\set}[1]{\left\lbrace #1\right\rbrace}
\providecommand{\abs}[1]{\left\lvert#1\right\rvert}
\providecommand{\norm}[1]{\left\lVert#1\right\rVert}
\newcommand{\qtq}[1]{\quad\text{#1}\quad}
\newcommand{\scalar}[1]{\left\langle#1\right\rangle}

\DeclareMathOperator{\sgn}{sign}
\DeclareMathOperator{\real}{Re}

\newtheorem{theorem}{Theorem}
\newtheorem{proposition}[theorem]{Proposition}
\newtheorem{lemma}[theorem]{Lemma}

\theoremstyle{definition}
\newtheorem*{definition}{Definition}

\theoremstyle{remark}
\newtheorem{remark}[theorem]{Remark}
\newtheorem{remarks}[theorem]{Remarks}
\newtheorem{example}[theorem]{Example}

\begin{document}
\title{Representations of dual spaces}

\author[T. Delzant]{Thomas Delzant}
\address{UFR de mathématique et d'informatique, Université de Strasbourg, 7 rue René Descartes\\
         67084 Strasbourg Cedex, France}
\email{delzant@math.unistra.fr}
\author[V. Komornik]{Vilmos Komornik}
\address{UFR de mathématique et d'informatique, Université de Strasbourg, 7 rue René Descartes\\
         67084 Strasbourg Cedex, France}
\email{komornik@math.unistra.fr}
\date{Version 2019-02-18}

\begin{abstract}
We give a nonlinear representation of the duals for a class of Banach spaces.
This leads to classroom-friendly proofs of the classical  representation theorems $H'=H$ and $(L^p)'=L^q$.
Our proofs extend to a family of Orlicz spaces, and yield as an unexpected byproduct a version of the Helly--Hahn--Banach theorem.
\end{abstract}
\maketitle

\section{A bijection between normed spaces and their duals}\label{s1}

Let $X$ be a real normed space and $X'$ its dual.
We recall the following two notions:

\begin{definition}\mbox{}
\begin{itemize}
\item $X$ is \emph{uniformly convex} if for every $\eps>0$ there exists a $\delta>0$ such that
\begin{equation}\label{1}
\qtq{if}\norm{x}=\norm{y}=1\qtq{and}\norm{\frac{x+y}{2}}\ge 1-\delta,\qtq{then}\norm{x-y}\le \eps.
\end{equation}

\item A function $F:X\to\RR$ is \emph{G\^ateaux differentiable} in $x$ if there exists a continuous linear functional $F'(x)\in X'$ such that for all $u\in X$,
\begin{equation*}
F(x+tu)=F(x)+tF'(x)u+o(t)\qtq{as}t\to 0.
\end{equation*}
\end{itemize}
\end{definition}

Let us denote by $S$ and $S'$ the unit spheres of $X$ and $X'$, respectively.
We recall (see Lemma \ref{l3} below) that if $X$ is a uniformly convex Banach space, then the restriction to $S$ of every linear form $y\in S'$ achieves its maximum in a unique point $M(y)$; furthermore, the map $M:S'\to S$ is continuous.

\begin{theorem}\label{t1}
Let $X$ be a uniformly convex Banach space. 
If its norm $N:X\to\RR$ is G\^ateaux differentiable on the unit sphere $S$, then   $M:S'\to S$ is a continuous \emph{bijection}, and it inverse is the derivative $N':S\to S'$ of the norm.
\end{theorem}

\begin{proof}
If $x,z\in S$ and $t>0$, then by the triangular inequality we have
\begin{equation*}
\frac{N(x+tz)-N(x)}{t}
\le \frac{N(tz)}{t}
=N(z)=1,
\end{equation*}
so that $N'(x)z\le 1$, with equality if $z=x$. 
It follows that $\norm{N'(x)}=1$.

By uniform convexity, the linear form $N'(x)$ achieves its maximum at a unique point (see Lemma \ref{l3} below), so that $N'(x)z<1$ for all $z\in S$, different from $x$.
Hence every $x\in S$ is uniquely determined by $N'(x)$, and thus $N'$ is injective.

To prove that $N'$ is onto, pick $y\in S'$  arbitrarily.
Then $x:=M(y)$ is a point of minimal norm of the affine hyperplane $\set{z\in X\ :\ y(z)=1}$, and therefore
\begin{equation*}
\frac{d}{dt}\norm{x+t(z-y(z)x)}|_{t=0}=0,\qtq{i.e.,}
N'(x)(z-y(z)x)=0
\end{equation*}
for every $z\in X$.
Since $N'(x)x=1$, this is equivalent to 
$y=N'(x)$.
\end{proof}

\begin{remarks}\label{r2}\mbox{}
\begin{enumerate}[\upshape (i)]
\item The above proof is essentially an application of the Lagrange multiplier theorem to maximize a linear functional on the unit sphere.
Since the norm is not assumed to be a $C^1$ function, we have considered an equivalent extremal problem to minimize the norm on a closed affine hyperplane. 
\item The homogeneous extension of $N':S\to S'$ yields a bijection between $X$ and $X'$.
\item It follows from the theorem that if $N'$ is continuous on $S$, then $M$ is a homeomorphism between $S'$ and $S$, and its homogeneous extension is a homeomorphism between $X'$ and $X$.
\item If $N'$ is continuous on $S$, then it is also continuous on $X\setminus\set{0}$ by homogeneity.
Then the norm  is in fact Fr\'echet differentiable in every $a\in X\setminus\set{0}$, and hence it is a $C^1$ function. 
To see this, for any fixed $\eps>0$ choose a small open ball $B_r(a)\subset X\setminus\set{0}$ such that $\norm{N'(x)-N'(a)}<\eps$ for all $x\in B_r(a)$.
For any fixed $u\in B_r(0)$, applying the Lagrange mean value theorem to the function 
\begin{equation*}
t\mapsto N\left(a+tu\right)-tN'(a)u
\end{equation*}
in $[0,1]$ we obtain that 
\begin{equation*}
\norm{N\left(a+u\right)-N(a)-N'(a)u}\le \eps\norm{u}.
\end{equation*}
\end{enumerate}
\end{remarks}

We recall the short proof of the lemma used above:

\begin{lemma}\label{l3}
Let $X$ be a uniformly convex Banach space, and $y\in S'$. There exists a (unique) point $x=M(y)\in S$ such that
\begin{equation*}
y(x)=1,\qtq{and}y(z)<1\qtq{for all other}z\in S.
\end{equation*} 
Furthermore, the map $M:S'\to S$ is a continuous.
\end{lemma}

\begin{proof}
Since $X$ is complete and $S$ is closed, it suffices to show that every sequence $(x_n)\subset S$ satisfying $y(x_n)\to 1$ is a Cauchy sequence.
Given $\eps>0$ arbitrarily, choose $\delta>0$ according to the uniform convexity, and then choose a large integer $n_0$ such that $y(x_n)>1-\delta$ for all $n\ge n_0$.
If $m,n\ge n_0$, then 
\begin{equation*}
\norm{\frac{x_n+x_m}{2}}\ge y\left(\frac{x_n+x_m}{2}\right)>1-\delta
\end{equation*}
and therefore $\norm{x_n-x_m}<\eps$.
\medskip

For the continuity of $M$ we have to show that if $y_n\to y$ in $S'$, then $M(y_n)\to M(y)$.
Writing $x_n:=M(y_n)$ and $x:=M(y)$ for brevity, since
\begin{equation*}
\abs{y(x_n)-1}=\abs{y(x_n)-y_n(x_n)}\le\norm{y_n-y}\to 0,
\end{equation*}
$y(x_n)\to 1$, and therefore $(x_n)$ converges to a point $\tilde x\in S$ satisfying $y(\tilde x)=1$ as in the first part of the proof.
Since $y(x)=1$ we have $\tilde x=x$ by uniqueness.
\end{proof}

\section{Applications of Theorem \ref{t1}}\label{s2}

Theorem \ref{t1} has important consequences.
We start with a strengthened version of the Helly--Hahn--Banach theorem for special spaces:

\begin{theorem}\label{t4}
Let $X$ be a Banach space satisfying the hypotheses of Theorem \ref{t1}, and $y_1$ a continuous linear functional defined on some subspace $X_1$ of $X$.

Then $y_1$ extends to a \emph{unique} continuous linear functional $y\in X'$ with preservation of the norm.
\end{theorem}

\begin{proof}
We may assume that $\norm{y_1}=1$, and we may assume by a continuous extension to the closure of $X_1$ that $X_1$ is closed.
Then $X_1$ also satisfies the hypotheses of Theorem \ref{t1}, so that $y_1=N'(x_1)|_{X_1}$ for a unique $x_1\in S\cap X_1$. 
It follows that $y:=N'(x_1)$ is an extension of $y_1$, and $\norm{N'(x_1)}=1=\norm{y_1}$.

If $\tilde y\in S'$ is an arbitrary extension of $y_1$, then $\tilde y=N'(\tilde x_1)$ for a \emph{unique} $\tilde x_1\in S$, characterized by the equality $\tilde y(\tilde x_1)=1$.
Since $x_1\in S$ and $\tilde y(x_1)=y_1(x_1)=1$, $\tilde x_1=x_1$ by uniqueness. 
\end{proof}

\begin{remark}\label{r5}
Geometrically the proof is based on the observation that the unique hyperplane separating $x_1$ from the unit ball is the tangent hyperlane at $x_1$.
\end{remark}

We know that every uniformly convex Banach space is reflexive.
This can be seen easily under the further assumption that $X'$ satisfies the hypotheses of Theorem \ref{t1}:

\begin{proposition}\label{p6}
If $X, X'$ are uniformly convex Banach spaces and the norm of $X'$ is G\^ateaux differentiable on $S'$, then $X$ is reflexive.
\end{proposition}

\begin{proof}
Given $\Phi\in X''$ arbitrarily, we have to find $x\in X$ satisfying $\Phi(y)=y(x)$ for all $y\in X'$.

We may assue that $\norm{\Phi}=1$.
By Lemma \ref{l3} there exists a unique $y\in S'$ satisfying $\Phi(y)=1$, and then a unique $x\in S$ satisfying $y(x)=1$.
Defining $\Phi_x\in X''$ by $\Phi_x(\tilde y):=\tilde y(x)$ for all $\tilde y\in X'$, we have $\Phi, \Phi_x\in S''$ and $\Phi(y)=1=\Phi_x(y)$. 
Applying Theorem \ref{t1} to $X'$ we conclude that $\Phi= \Phi_x$.
\end{proof}

Now we turn to the description of the duals of Hilbert and $L^p$ spaces.

\begin{theorem}[Riesz--Fr\'echet \cite{A2Fre1907,A2Rie1907C4}]\label{t7}
If $X$ is a Hilbert space with scalar product $\scalar{\cdot,\cdot}$, then the formula 
\begin{equation*}
\Phi(x)u:=\scalar{x,u},\quad x,u\in X
\end{equation*} 
defines an isometric isomorphism $\Phi$ of $X$ onto $X'$.
\end{theorem}

\begin{proof}
Since $\Phi$ is an isometric isomorphism of $X$ \emph{into} $X'$ by the Cauchy--Schwarz inequality, it suffices to show that $\Phi$ maps $S$ \emph{onto} $S'$.

Every Hilbert space is uniformly convex by the parallelogramma law, and its norm is G\^ateaux differentiable in every $x\in S$ with $N'(x)=\Phi(x)$ because the following relation holds for each $u\in X$:
\begin{equation*}
\norm{x+tu}
=\sqrt{\norm{x}^2+2t\scalar{x,u}+t^2\norm{u}^2}
=1+t\scalar{x,u}+o(t)\qtq{as}t\to 0.
\end{equation*} 
Applying Theorem \ref{t1} we conclude that $\Phi$ is a bijection between $S$ and $S'$.
\end{proof}

As we shall see in the more general context of Orlicz spaces (Lemma \ref{l12}), the $L^p$ norms satisfy the hypotheses of Theorem \ref{t1}.

\begin{theorem}[Riesz \cite{A2Rie1910C10,A2RieNag}]\label{t8}
If $X=L^p$ with $1<p<\infty$ on some measure space and $q=p/(p-1)$ is the conjugate exponent, then the formula
\begin{equation*}
\Phi(g)f:=\int gf\ dx,\quad g\in L^q,\quad f\in L^p
\end{equation*} 
defines a linear isomorphism $\Phi$ of $L^q$ onto $(L^p)'$.
\end{theorem}

\begin{proof}
Since $\Phi$ is an isometric isomorphism of $L^q$ \emph{into} $(L^p)'$ by the H\"older inequality, it suffices to show that each $\varphi\in (L^p)'$ of norm one has the form $\varphi=\Phi(g)$ with some $g\in L^q$.

Due to Lemma \ref{l12}, the hypotheses of Theorem \ref{t1} are fulfilled, and the G\^ateaux derivative of the norm of $L^p$ in any $h_0\in S$ is given by the formula 
\begin{equation*}
N'(h_0)u=\int \abs{h_0}^{\frac{p}{q}}(\sgn h_0)u\ dx.
\end{equation*}
This implies $\varphi=\Phi(g_0)$ with $g_0=\abs{h_0}^{\frac{p}{q}}\sgn(h_0)$.
\end{proof}

\begin{remarks}\label{r9}\mbox{}
\begin{enumerate}[\upshape (i)]
\item The formula $F(h):=\abs{h}^{\frac{p}{q}}\sgn h$ defines a bijection between $L^p$ and $L^q$ whose inverse is $F^{-1}(g):=\abs{g}^{\frac{q}{p}}\sgn g$.
Since $\Phi=N'\circ F^{-1}$ and $N', \Phi$ are homeomorphisms by Theorems \ref{t1}, \ref{t8} and Remark \ref{r2} (iii), we conclude that $F$ is a homeomorphism between $L^p$ and $L^q$.
This is a special case of Mazur's theorem \cite{Mazur1929} stating that the spaces $L^p$ are homeomorphic for all finite $p$.

\item We recall the extensions of Theorems \ref{t7} and \ref{t8} to complex Hilbert and $L^p$ spaces: the formulas
\begin{equation*}
\Phi(g):=\scalar{g,f}
\qtq{and}
\Phi(g)f:=\int \overline{g} f\ dx
\end{equation*}
define \emph{conjugate} linear isometries of of $X$ onto $X'$ and of $L^q$ onto $(L^p)'$, respectively.

Only the surjectivity needs an additional argument, and this follows by the classical method of Murray \cite{Murray1936}.
For example, if $\varphi\in (L^p)'$, then $\Phi(g)=\varphi$ with $g:=g_1+ig_2$ by a direct computation, where the real-valued functions $g_1, g_2\in L^q$ are defined by the equalities
\begin{equation*}
\real\varphi(h)=\int g_1h\ dx\qtq{and}\real\varphi(ih)=\int g_2h\ dx
\end{equation*}
for all real-valued functions $h\in L^p$.
\item Another proof of Theorem \ref{t8} was given recently in \cite{Shioji2018}.
\end{enumerate}
\end{remarks}

\section{Duals of Orlicz spaces}\label{s3}

In this section we generalize the proof of Theorem \ref{t8} to a class of Orlicz spaces \cite{KraRut1961,Luxemburg1955,
Orlicz1932,Orlicz1936,Zaanen1983}.
First we briefly recall some basic facts.
Let $P:\RR\to\RR$ be an even convex function satisfying the conditions
\begin{equation*}
\lim_{u\to 0}\frac{P(u)}{u}=0
\qtq{and}
\lim_{u\to \infty}\frac{P(u)}{u}=\infty,
\end{equation*}
and $Q:\RR\to\RR$ its convex conjugate, defined by the formula
\begin{equation*}
Q(v):=\sup_{u\in\RR}\set{u\abs{v}-P(u)}.
\end{equation*}
For example, if $P(u)=p^{-1}\abs{u}^p$ for some $1<p<\infty$, then $Q(v)=q^{-1}\abs{v}^q$ with the conjugate exponent $q=p/(p-1)$.

Given a measure space $(\Omega,\mm,\mu)$, henceforth we  consider only measurable functions $f,g,h:\Omega\to\RR$.
The \emph{Orlicz spaces} $L^P$ and $L^Q$ are defined by
\begin{align*}
f\in L^P&\Longleftrightarrow fg\qtq{is integrable whenever}\int Q(g)\ dx<\infty
\intertext{and}
g\in L^Q&\Longleftrightarrow fg\qtq{is integrable whenever}\int P(f)\ dx<\infty.
\end{align*}
If we endow $L^P$ with the \emph{Luxemburg norm}
\begin{equation*}
\norm{f}_{P,\text{Lux}}:=\inf\set{k>0\ :\ \int P(k^{-1}f)\ dx\le 1}
\end{equation*}
and $L^Q$ with the \emph{Orlicz norm}
\begin{equation*}
\norm{g}_{Q,\text{Orl}}:=\sup\set{\int fg\ dx\ :\ \int P(f)\ dx\le 1},
\end{equation*}
then they become Banach spaces, 
\begin{equation*}
\int \abs{fg}\ dx\le \norm{f}_{P,\text{Lux}}\norm{g}_{Q,\text{Orl}}
\end{equation*}
for all $f\in L^P$ and $g\in L^Q$ because $\norm{f}_{P,\text{Lux}}\le 1\Longleftrightarrow \int P(f)\ dx\le 1$, and the formula
\begin{equation*}
\Phi(g)(f):=\int fg\ dx
\end{equation*}
defines a linear isometry $\Phi:L^Q\to (L^P)'$.

\begin{theorem}\label{t10}
Assume that $P$ is  differentiable, and satisfies the following conditions:
\begin{enumerate}[\upshape (i)]
\item there exists a constant $\alpha$ such that $P(2u)\le \alpha P(u)$ for all $u\in\RR$;
\item for each $\eps>0$ there exists a $\rho(\eps)>0$ such that 
\begin{equation*}
\abs{u-v}\ge\eps(\abs{u}+\abs{v})\Longrightarrow 
\frac{P(u)+P(v)}{2}-P\left(\frac{u+v}{2}\right)\ge\rho(\eps)\frac{P(u)+P(v)}{2}.
\end{equation*}
\end{enumerate}
Then $\Phi:L^Q\to (L^P)'$ is an isometrical isomorphism of $L^Q$ \emph{onto} $(L^P)'$.
\end{theorem}

\begin{example}\label{e11}
If $1<p<\infty$ and $P(u):=p^{-1}\abs{u}^p$, then Theorem \ref{t10} reduces to Riesz's theorem $(L^p)'=L^q$ with $q=p/(p-1)$.
Indeed, $P$ is differentiable with $P'(u)=\abs{u}^{p-1}\sgn u$, and the condition (i) is satisfied with $\alpha=2^p$.
To check (ii) we may assume by homogeneity that $(u,v)$ belongs to the compact set
\begin{equation*}
K:=\set{(u,v)\in\RR^2\ :\ \abs{u}+\abs{v}=1}.
\end{equation*}
The inequality follows by observing that the left hand side has a positive minimum on $K$ by continuity and strict convexity, while  $P(u)+P(v)$ has a finite maximum here.
Finally, on the classical spaces $L^p$ the Orlicz and Luxemburg norms coincide.
\end{example}

For the proof of Theorem \ref{t10} we admit temporarily the following

\begin{lemma}\label{l12}
Assume that the conditions of Theorem \ref{t10} are satisfied.
Then
\begin{enumerate}[\upshape (i)]
\item $L^P$ is uniformly convex;
\item $P'(h)\in L^Q$ for every $h\in L^P$;
\item for any fixed $f,h\in L^P$ the function $t\mapsto \int P(h+tf)\ dx$ is differentiable in zero, and its derivative is equal to $\int P'(h)f\ dx$.
\end{enumerate}
\end{lemma}

\begin{proof}[Proof of Theorem \ref{t10}]
As before, it suffices to show that every 
$\varphi\in (L^P)'$ of unit norm is of the form $\Phi(h)$ for a suitable $h\in L^P$.

As in the proof of Theorem \ref{t1}, by the uniform convexity of $L^P$ there exists a function $h\in L^P$ such that $\varphi(h)=1$, and $\norm{f}>1$ for all other functions $f$ in the closed affine hyperplane $H:=\set{f\in L^P\ :\ \varphi(f)=1}$.

By the definition of the Luxemburg norm we have
\begin{equation*}
\int P(h)\ dx=1, \qtq{and}\int P(f)\ dx>1\qtq{for all other}f\in H,
\end{equation*}
so that $h$ also minimizes the functional $f\mapsto \int P(f)\ dx$ on $H$.
Hence for any fixed $f\in L^P$ the  function 
\begin{equation*}
t\mapsto \int P\left(h+t[f-\varphi(f)h]\right)\ dx
\end{equation*}
has a minimum in $0$, and therefore its derivative vanishes here:
\begin{equation*}
\int P'(h)\left(f-\varphi(f)h\right)\ dx=0.
\end{equation*}
This is equivalent to $\varphi=\varphi_g$ with 
\begin{equation*}
g=\frac{P'(h)}{\int P'(h)h\ dx}\in L^Q.\qedhere
\end{equation*}
\end{proof}

\begin{proof}[Proof of Lemma \ref{l12}]
(i) Following McShane \cite{McShane1950} first we prove for all $\eps\in (0,1)$ the inequality
\begin{equation}\label{2}
P\left(\frac{u-v}{2}\right)\le\eps\cdot\frac{P(u)+P(v)}{2}
+\frac{1}{\rho(\eps)}\left(\frac{P(u)+P(v)}{2}-P\left(\frac{u+v}{2}\right)\right).
\end{equation}
In case $\abs{u-v}\le \eps(\abs{u}+\abs{v})$ this readily follows from the convexity and evenness of $P$ and from the equality $P(0)=0$ because
\begin{equation*}
P\left(\frac{u-v}{2}\right)
\le P\left(\frac{\eps}{2}\abs{u}+\frac{\eps}{2}\abs{v}+(1-\eps)\cdot 0\right)
\le\eps\cdot\frac{P(u)+P(v)}{2}.
\end{equation*}
If $\abs{u-v}\ge \eps(\abs{u}+\abs{v})$, then we infer from the convexity of $P$ and from the condition (ii) of Theorem \ref{t10} that
\begin{equation*}
P\left(\frac{u-v}{2}\right)
\le \frac{P(u)+P(v)}{2}
\le \frac{1}{\rho(\eps)}\left(\frac{P(u)+P(v)}{2}-P\left(\frac{u+v}{2}\right)\right).
\end{equation*} 
\medskip 

We complete the proof of the lemma by showing that if
\begin{equation*}
\int P(f)\ dx\le 1,\quad 
\int P(g)\ dx\le 1\qtq{and}
\int P\left(\frac{f+g}{2}\right)\ dx\ge 1-\eps \rho(\eps),
\end{equation*}
then 
\begin{equation*}
\int P(f-g)\ dx\le 2\alpha\eps
\end{equation*}
with $\alpha$ given by the condition (i) of Theorem \ref{t10}. 
This follows by using the inequality \eqref{2}:
\begin{align*}
\frac{1}{\alpha}\int P(f-g)\ dx
&=\int P\left(\frac{f-g}{2}\right)\\
&\le\eps\cdot\int \frac{P(f)+P(g)}{2}\ dx
+\frac{1}{\rho(\eps)}\int \frac{P(f)+P(g)}{2}-P\left(\frac{f+g}{2}\right)\ dx\\
&\le \eps+\frac{1}{\rho(\eps)}\left(1-\int P\left(\frac{f+g}{2}\right)\ dx\right)\\
&\le \eps+\frac{1-[1-\eps \rho(\eps)]}{\rho(\eps)}
=2\eps.
\end{align*}
\medskip 

(ii) By definition we have to show that $P'(h)f$ is integrable for every $f\in L^P$.
Setting $g:=\abs{h}+\abs{h}\in L^P$, this follows from the inequalities
\begin{equation*}
\int\abs{P'(h)f}\ dx
\le \int\abs{P'(g)g}\ dx
\le \alpha \int P(g)\ dx<\infty.
\end{equation*}
The last inequality is a consequence of the condition (i) of Theorem \ref{t10}: we have
\begin{equation*}
uP'(u)\le \int_u^{2u}P'(t)\ dt=P(2u)-P(u)\le\alpha P(u)
\end{equation*} 
for all $u\ge 0$, and $uP'(u)$ is an even function.
(The equality holds because, as a convex function, $P$ is absolutely continuous.)
\medskip 

(iii) We have
\begin{equation*}
\lim_{t\to 0}\frac{P(f+th)-P(f)}{t}
=P'(f)h
\end{equation*}
almost everywhere (where both $f$ and $h$ are defined and are finite).
Furthermore, if $0<\abs{t}\le 1$, then  
\begin{equation*}
\abs{\frac{P(f+th)-P(f)}{t}}=\abs{P'(f+t'h)h}
\le \abs{P'(g)g}
\end{equation*}
with $g:=\abs{f}+\abs{h}\in L^P$ by the Lagrange mean value theorem with some $t'$ between $0$ and $t$ (depending on $x$).
The last function is integrable because
\begin{equation*}
\int \abs{P'(g)g}\ dx\le \norm{g}_{P,\text{Lux}}\norm{P'(g)}_{Q,\text{Orl}}<\infty.
\end{equation*} 
Applying Lebesgue's dominated convergence theorem we conclude that
\begin{equation*}
\lim_{t\to 0}\frac{\int P(f+th)\ dx-\int P(f)\ dx}{t}
=\int P'(f)h\ dx.\qedhere
\end{equation*}
\end{proof}

\begin{remark}\label{r13}
A complete, but necessarily rather technical characterization of uniformly convex Orlicz space was given by  different tools in  \cite{Milnes1957}.
\end{remark}


\begin{thebibliography}{9}


\bibitem{A2Fre1907} M.\ Fréchet, {\em Sur les ensembles de fonctions et les
opérations linéaires}, C.\ R.\ Acad.\ Sci.\ Paris 144 (1907), 1414--1416.

\bibitem{KraRut1961}
M. A. Krasnosel'skii, Ja. V. Rutickii, 
\emph{Convex Functions and Orlicz Spaces,}
Noordhoff, Groningen, 1961.

\bibitem{Luxemburg1955}
W. A. J. Luxemburg, 
\emph{Banach Function Spaces}, Delft, 1955.

\bibitem{Mazur1929}
S. Mazur,
\emph{Une remarque sur l'homéomorphie des champs fonctionnels,}
Studia Math. 1 (1929), 83--85.

\bibitem{McShane1950} E. J. McShane, {\em Linear functionals on certain Banach spaces},
Proc. Amer.  Math. Soc. 1 (1950), 402--408.

\bibitem{Milnes1957}
H. W. Milnes, 
\emph{Convexity of Orlicz spaces,}
Pacific J. Math. 7 (1957), 1451--1483.

\bibitem{Murray1936} 
F. J.\ Murray,  
\emph{Linear transformations in $L^p$, $p>1$,}
Trans. Amer. Math. Soc. 39 (1936), no. 1, 83--100.

\bibitem{Orlicz1932}
W. Orlicz, 
\emph{Über eine gewisse Klasse von Räumen vom Typus B},
Bull. Int. Acad. Polon. Sci. A 8/9 (1932), 207--220

\bibitem{Orlicz1936}
W. Orlicz, 
\emph{Über Räume $(L^M)$},
Bull. Int. Acad. Polon. Sci. A  (1936), 93--107.


\bibitem{A2Rie1907C4} F.\ Riesz, {\em Sur une espèce de géométrie
analytique des systèmes de fonctions som\-ma\-bles},
C.\ R.\ Acad.\ Sci.\ Paris 144 (1907), 1409--1411; \cite{A2Riesz} I, 386--388.

\bibitem{A2Rie1910C10} F. Riesz, {\em Untersuchungen über Systeme
integrierbar Funktionen}, Math. Ann. 69 (1910), 449--497; \cite{A2Riesz} I, 441--489.


\bibitem{A2Riesz} F. Riesz,
{\em Oeuvres complètes I-II} (ed. Á. Császár),
Akadémiai Kiadó, Budapest, 1960.


\bibitem{A2RieNag} 
F. Riesz, B. Sz.-Nagy,
{\em Le\c cons d'analyse fonctionnelle},
Akadémiai Kiadó, Budapest, 1952.
English translation:
\emph{Functional Analysis,} Dover,  1990.

\bibitem{Shioji2018}
N. Shioji,
\emph{Simple Proofs of the Uniform Convexity of $L^p$ and the Riesz Representation Theorem for $L^p$,}
Amer. Math. Monthly 125 (2018), no. 8, 733--738.


\bibitem{Zaanen1983}
A. C. Zaanen,
\emph{Riesz Spaces II},
North-Holland, Amsterdam, 1983.

\end{thebibliography}
\end{document}